\documentclass[12pt,a4paper]{article}
\usepackage{latexsym,amssymb,amsmath,amsthm}
\usepackage{enumitem}
\usepackage{xcolor}
\usepackage{graphicx}
\usepackage{subfig}

\newtheorem{theorem}{Theorem}[section]
\newtheorem*{theorem*}{Theorem}
\newtheorem{lemma}[theorem]{Lemma}

\newtheorem{observation}[theorem]{Observation}

\newtheorem{conjecture}[theorem]{Conjecture}
\newcounter{cclaim}
\newtheorem{claim}[cclaim]{Claim}

\newcommand{\claimproofend}{\hspace*{.1mm}\hspace{\fill}}
\newenvironment{claimproof}{}{\claimproofend\par\vspace{2mm}}

\newcommand{\PP}{\mathcal P}

\newcommand{\SSS}{\mathcal S}

\newcommand{\fig}[1]{\includegraphics[page=#1]{T-conn-figs}}

\newcommand{\size}[1]{\left|#1\right|}

\newcommand\Setx[1] {\left\{{#1}\right\}}

\newcommand{\bdx}[2]{\partial_{#1}({#2})}

\newcommand{\AAA}A%{\textsc{a}}
\newcommand{\BB}B%{\textsc{b}}
\newcommand{\CC}C%{\textsc{c}}

\begin{document}
\title{\textbf{Packing $T$-connectors in graphs needs more connectivity}}
\author{Roman \v{C}ada$^{\:1}$ \and Adam Kabela$^{\:1}$ \and Tom\'a\v
  s Kaiser$^{\:1}$\and Petr Vr\'{a}na$^{\:1}$}

\date{}

\maketitle
\begin{abstract}
  Strengthening the classical concept of Steiner trees, West and Wu
  [J. Combin. Theory Ser. B 102 (2012), 186--205] introduced the
  notion of a $T$-connector in a graph $G$ with a set $T$ of
  terminals. They conjectured that if the set $T$ is
  $3k$-edge-connected in $G$, then $G$ contains $k$ edge-disjoint
  $T$-connectors. We disprove this conjecture by constructing
  infinitely many counterexamples for $k=1$ and for each even $k$.
\end{abstract}
\footnotetext[1]{Department of Mathematics and European Centre of
  Excellence NTIS (New Technologies for the Information Society),
  University of West Bohemia, Pilsen, Czech Republic. E-mail:
  \texttt{\{cadar,kabela,kaisert,vranap\}@kma.zcu.cz}. Supported by
  project GA20-09525S of the Czech Science Foundation.}%

% Connectivity:
% \begin{enumerate}
% \item prove things with half-edges
% \item for $F_1,\dots,F_{\ell-1}$: $3k$-connectedness simple
% \item $F^*_\ell$: cuts of size $2k$, $5k/2$ but can be completed (by
%   half-edges?) to $3k$
% \item substitute $F^*_\ell$ into mega-graph: big cubic [NH bipartite
%   cubic] $\to$ $k$-fold edges, which is $3k$-connected, mind
%   connection of $F^*_\ell$ half-edges 
% \end{enumerate}

\section{Introduction}
\label{sec:introduction}

Let $G$ be a graph and $T \subseteq V(G)$ a set of vertices called
\emph{terminals}. The vertices in $V(G)\setminus T$ are the
\emph{non-terminals}. For brevity, we may call the pair $(G,T)$ a
\emph{graph with terminals}.

A \emph{$T$-tree} (or \emph{Steiner tree}) in $G$ is a tree whose
vertex set includes every vertex from $T$. Kriesell~\cite{K}
investigated the existence of edge-disjoint $T$-trees in relation to
the connectivity of $T$ in $G$, defined as follows. The set $T$ is
said to be \emph{$k$-edge-connected} in $G$ if $G$ contains no
edge-cut of size smaller than $k$ separating two vertices from
$T$. Note that this entails no assumption about the connectivity of
$G$ itself. Kriesell~\cite{K} conjectured that if $T$ is
$2k$-edge-connected in $G$, then $G$ contains $k$ edge-disjoint
$T$-trees. This conjecture is still open but DeVos, McDonald and
Pivotto~\cite{DMP} proved that the conclusion holds if $T$ is
$(5k+4)$-connected in $G$ (see also \cite{L,WW} for earlier results).

In this paper, we deal with related objects called $T$-connectors,
introduced by West and Wu~\cite{WW}. Recall first that a
\emph{$T$-path} in $G$ is a path whose both endvertices belong to $T$
but which is otherwise disjoint from $T$. The operation of
\emph{short-cutting} a nontrivial path $P$ consists in removing all
edges of $P$ and adding an edge joining its endvertices. A
\emph{$T$-connector} in $G$ is the union of a family of edge-disjoint
$T$-paths in $G$ such that if we shortcut the $T$-paths one by one, we
obtain a graph whose induced subgraph on $T$ is connected. Improving
upon a result of~\cite{WW}, DeVos et al.~\cite{DMP} proved that $G$
has $k$ edge-disjoint $T$-connectors provided that $T$ is
$(6k+6)$-edge-connected in $G$.

On the other hand, West and Wu~\cite{WW} constructed examples of
graphs with terminals $(G,T)$ such that $T$ is $(3k-1)$-edge-connected
in $G$, but $G$ does not admit $k$ edge-disjoint $T$-connectors. They
conjectured that higher edge-connectivity of $T$ already implies the existence
of $k$ edge-disjoint $T$-connectors.
\begin{conjecture}\label{conj:west-wu}
  Let $k$ be a positive integer and $(G,T)$ a graph with terminals. If
  $T$ is $3k$-edge-connected in $G$, then $G$ admits $k$ pairwise
  edge-disjoint $T$-connectors.
\end{conjecture}

We disprove Conjecture~\ref{conj:west-wu}. Surprisingly, it turns out
that there is a counterexample already for the case $k=1$. More
generally, we provide an infinite family of counterexamples for $k=1$
and for each even $k$. Our main result is the following.
\begin{theorem}\label{t:main}
  For any nonnegative integer $\ell$, there are infinitely many graphs
  with terminals $(G,T)$ such that
  \begin{enumerate}[label=(\arabic*)]
  \item $T$ is $(3\cdot 2^\ell)$-edge-connected in $G$, and
  \item $G$ does not admit $2^\ell$ edge-disjoint $T$-connectors.
  \end{enumerate}
\end{theorem}

The plan of the paper is as follows. In
Section~\ref{sec:preliminaries}, we recall some necessary notions and
prove a useful technical lemma on gluing graphs with terminals
(Lemma~\ref{l:conn}). Section~\ref{sec:construction} describes the
construction of the counterexamples $(G_\ell,T_\ell)$ for the case
where $k=2^\ell$ is a power of two. The same assumption that $k$ is a
power of two is used in Sections~\ref{sec:choosing} and
\ref{sec:proof}. In Section~\ref{sec:choosing}, we show that if
$G_\ell$ contains $k$ edge-disjoint $T_\ell$-connectors, then these
$T_\ell$-connectors can be chosen so as to satisfy certain additional
properties in relation to the edge-cuts in $G_\ell$. In
Section~\ref{sec:proof}, we prove the main result. Finally, we discuss
in Section~\ref{sec:even} how the construction is adapted to the case
where $k$ is an even integer which is not a power of two.

%%%%%%%%%%%%%%%%%%%%%%%%%%%%%%%%%%%%%%%%%%%%%%%%%%%%%%%%%%%%%%%%%%%%%%

\section{Preliminaries}
\label{sec:preliminaries}

If $n$ is a positive integer, we let $[n]$ denote the set
$\Setx{1,\dots,n}$. Throughout this paper, parallel edges are allowed
in graphs. A path $P$ in a graph $G$ is said to be an \emph{$uv$-path}
if its endvertices are $u$ and $v$.

If $G$ is a graph and $X\subseteq V(G)$, then the symbol $\bdx G X$
denotes the set of all edges of $G$ with exactly one endvertex in
$X$. The \emph{degree} of $X$ in $G$, denoted by $d_G(X)$, is defined
as $\size{\bdx G X}$. We abbreviate $d_G(\Setx v)$ to $d_G(v)$.

In order to break down the construction into steps, we define an
insertion operation as follows. Let $(G_1,T_1)$ and $(G_2,T_2)$ be
disjoint graphs with terminals, let $t_1\in T_1$ and $t_2\in T_2$ be
terminals of equal degree, and let
$h:\,\bdx{G_1}{t_1}\to\bdx{G_2}{t_2}$ be a bijection. The
\emph{insertion} of $G_2$ into $G_1$ via $h$ is the graph with
terminals $G_1\circ_h G_2$ obtained as follows:
\begin{itemize}
\item we start with the disjoint union of $G_1-t_1$ and $G_2-t_2$,
\item for each $e\in\bdx{G_1}{t_1}$, we add an edge joining the
  endvertex of $e$ distinct from $t_1$ to the endvertex of $h(e)$
  distinct from $t_2$,
\item the set of terminals in the resulting graph is
  $T_1\cup T_2 \setminus \Setx{t_1,t_2}$.
\end{itemize}
To keep things simple, we let the vertices $t_1$ and $t_2$ be implicit
in the symbol $h$ in $G_1\circ_h G_2$, and we do not include the
sets $T_1,T_2$ in the notation.

The following lemma helps us to control the connectivity of terminals
in graphs obtained by insertion.
\begin{lemma}
  \label{l:conn}
  Let $(G_1,T_1)$ and $(G_2,T_2)$ be disjoint graphs with
  terminals. For each $i$ in $\Setx{1,2}$, suppose that
  $\size{T_i} \geq 2$ and $T_i$ is $d$-edge-connected in $G_i$, and
  let $t_i\in T_i$ be a terminal of degree $d$. For a bijection
  $h:\,\bdx{G_1}{t_1}\to\bdx{G_2}{t_2}$, let $(G,T)$ denote the graph
  with terminals $G_1\circ_h G_2$. Then $T$ is $d$-edge-connected in
  $G$.
\end{lemma}
\begin{proof}
  For $v_1,v_2\in T$, we want to find $d$ edge-disjoint $v_1v_2$-paths
  in $G$. By symmetry, we may assume that $v_1\in T_1$. We discuss two
  cases.

  Suppose first that $v_2\in T_2$. Since $T_1$ is $d$-edge-connected
  in $G_1$, we can choose edge-disjoint $v_1t_1$-paths $P_1,\dots,P_d$
  in $G_1$. Similarly, we choose edge-disjoint $v_2t_2$-paths $Q_1,\dots,Q_d$ in
  $G_2$. For $i\in [d]$, let $e_i$ be the edge of $P_i$ incident with
  $t_1$ and consider the unique path $Q_j$ containing the edge
  $h(e_i)$. Glue $P_i$ and $Q_j$ by removing $t_1$ and $t_2$ and
  adding an edge of $G$ that joins the new endvertices of these
  shortened paths (that is, an edge of $G$ connecting the endvertices
  of $e_i$ and $h(e_i)$ other than $t_1$ and $t_2$). Note that this
  gives $d$ edge-disjoint $v_1v_2$-paths in $G$.

  For the second case, suppose that $v_2\in T_1$. Let $R_1,\dots,R_d$
  be edge-disjoint $v_1v_2$-paths in $G_1$. If none of them contains
  the vertex $t_1$, then $R_1,\dots,R_d$ are the sought paths in
  $G$. Thus, assume that some of the paths, say $R_1,\dots,R_c$,
  contain $t_1$. Note that $c\leq\frac d2$ since $t_1$ is an internal
  vertex for each of these edge-disjoint paths and $t_1$ has degree
  $d$ in $G_1$. Since $\size{T_2} \geq 2$, we may choose
  $w_2\in T_2\setminus\Setx{t_2}$ and choose a set $\SSS$ of $d$
  edge-disjoint $w_2t_2$-paths in $G_2$. For each $i\in [c]$, we
  combine $R_i$ with two paths from $\SSS$ and obtain a $v_1v_2$-path
  $R'_i$ in $G$ as follows.

  Let $i\in [c]$ and $e_{i,1},e_{i,2}$ be the two edges of $R_i$
  incident with $t_1$. For $j\in\Setx{1,2}$, let $S_{i,j}$ be the
  unique $w_2t_2$-path containing $h(e_{i,j})$ and let $s_{i,j}$ be
  the neighbour of $t_2$ on $S_{i,j}$. Although $S_{i,1}$ and
  $S_{i,2}$ are edge-disjoint, $S_{i,1}\cup S_{i,2}-t_2$ need not be a
  path in $G_2$ because of vertices of degree $4$. However, it is easy
  to see that $S_{i,1}\cup S_{i,2}$ contains an $s_{i,1}s_{i,2}$-path
  $S'_i$. Take the two parts of $R_i-t_1$ and the path $S'_i$, and glue
  them together by adding the edges that join the endvertices of $S'_i$
  and the corresponding vertices of $R_i-t_1$ to obtain a
  $v_1v_2$-path $R'_i$. The paths $R'_1,\dots,R'_c,R_{c+1},\dots,R_d$
  are edge-disjoint $v_1v_2$-paths in $G$ as required.
\end{proof}

%%%%%%%%%%%%%%%%%%%%%%%%%%%%%%%%%%%%%%%%%%%%%%%%%%%%%%%%%%%%%%%%%%%%%%

\section{The construction for $k$ a power of two}
\label{sec:construction}

Let $\ell\geq 0$ and set $k = 2^\ell$. We construct a graph $G_\ell$
and a set of terminals $T_\ell \subseteq V(G_\ell)$ satisfying the
properties in Theorem~\ref{t:main}. The construction is in several
steps. In each of the graphs with terminals we construct, all
terminals have degree $3k$ and all non-terminals have degree
$3$. Moreover, one of the terminals is usually designated as the
\emph{root} terminal $r$ and the edges incident with $r$ are
partitioned into three \emph{classes} denoted by $\AAA$, $\BB$,
$\CC$.

Let us first describe the construction for $\ell\geq 1$ and then
comment on the (much simpler) case $\ell = 0$.

\textbf{Step 1.} In steps 1--2, we describe the graphs with terminals
$(F_i,T_i)$, where $0\leq i \leq \ell-1$. The graph $F_0$ consists of
two terminals joined by $3k$ parallel edges. One of the terminals is
the root, and the partition of the edges into $A$, $B$, $C$ is
arbitrary, with $C$ containing a single edge (the \emph{output edge})
and the sizes of $A$ and $B$ being $\frac32k$ and $\frac32k-1$,
respectively.

\textbf{Step 2.} For $i > 0$, $(F_i,T_i)$ is shown in
Figure~\ref{fig:fi}. The partition of the edges incident with the root
is shown by edge labels $A$, $B$, $C$. The sizes of these classes are
$\frac32k$, $\frac32k-2^i$ and $2^i$, respectively. (Note that this is
consistent with the case $i=0$.) The edges of type $C$ are the
\emph{output edges} of $F_i$.

\textbf{Step 3.} The graph with terminals $(F^*_i,T^*_i)$
($0\leq i \leq \ell-1$) is constructed as follows. First,
$(F^*_0,T^*_0) = (F_0,T_0)$. Assume then that $i > 0$ and that
$(F^*_{i-1},T^*_{i-1})$ has already been constructed.

The graph with terminals $(F^*_i,T^*_i)$ is obtained by inserting one
copy of $(F^*_{i-1},T^*_{i-1})$ at the terminal $f_1$ of $F_i$, and
another copy at the terminal $f_2$. In the graph $F^*_{i-1}$, the
insertion involves the root terminal $r$. More precisely, let $h_1$ be
a bijection between $\bdx{F_i}{f_1}$ and $\bdx{F^*_{i-1}}r$ which maps
the edges between $f_1$ and $u_1^A$ to class $A$ edges, the edges
between $f_1$ and $u_1^B$ to class $B$ edges, and the remaining edges
incident with $f_1$ to class $C$ edges. (Note that the edge counts
match: $\frac32k$, $\frac32k-2^{i-1}$ and $2^{i-1}$, respectively.)
Let $h_2$ be an analogous bijection for $f_2$ (and $r$) in the place
of $f_1$ (and $r$), with $u_2^A$ playing the role of $u_1^A$ etc.

We define $(F^*_i,T^*_i)$ as the result of inserting $F^*_{i-1}$ into
$F_i$ at $f_1$ via $h_1$, followed by the insertion of $F^*_{i-1}$ at
$f_2$ via $h_2$. The root terminal of $F^*_i$ is defined as identical
with that of $F_i$, and so is its partition of the incident edges into
the classes $A, B, C$ and its set of \emph{output edges}.

\textbf{Step 4.} As the next step, we use the graph with terminals
$(F_\ell,T_\ell)$ shown in Figure~\ref{fig:fl} to construct
$(F^*_\ell,T^*_\ell)$. This involves six insertions of $F^*_{\ell-1}$
into $F_\ell$ at terminals from the set
\begin{equation*}
  \Setx{f_{12},f_{21},f_{23},f_{32},f_{31},f_{13}}.
\end{equation*}
For the insertion, say, at $f_{12}$, we choose any bijection mapping
the $\frac32k$ edges between $f_{12}$ and $w_1$ to class $A$ edges of
$F^*_{\ell-1}$, the $\frac k2$ edges between $f_{12}$ and the root of
$F_\ell$ to the class $C$ edges, and the remaining $k$ edges to class
$B$ edges. (For the other insertions, the choice of bijection is
similar, the only difference being that $w_1$ may be replaced by $w_2$
or $w_3$.) Note that $\frac k2 = 2^{\ell-1}$, so the bijection exists
indeed. The resulting graph with terminals is
$(F^*_\ell,T^*_\ell)$. Its root terminal is defined to be the same as
in $F_\ell$.

\textbf{Step 5.} In the last two steps, we construct the graph with
terminals $(G_\ell,T_\ell)$. Let $N$ be a $3$-connected
non-hamiltonian cubic bipartite graph~\cite{EH} (see also \cite{BGM}),
with colour classes $X$ and $Y$. Perform the following steps:
\begin{itemize}
\item construct a bipartite graph $N'$ by replacing each vertex
  $y\in Y$ with an independent set $y_1,\dots,y_k$ of non-terminals,
  each adjacent to the neighbours of $y$, and declaring the vertices
  in $X$ (now of degree $3k$) terminals,
\item insert one copy of $F^*_\ell$ at each vertex $x\in X$ in $N'$
  via a bijection $h$ between the set of $3k$ edges incident with $x$
  and the set of $3k$ edges incident with the root terminal of
  $F^*_\ell$ (which is the vertex of $F^*_\ell$ used for each
  insertion); $h$ is chosen such that for each neighbour $y$ of $x$ in
  $N$, all the $k$ edges of $N'$ corresponding to the edge $xy$ are
  mapped to the $k$ edges incident with $r$ and a vertex from
  $\{f_{pq},f_{qp}\}$, for suitable $1\leq p < q\leq 3$.
\end{itemize}

\textbf{Step 6.} Finally, insert a copy of $K_{3,3k}$ (with $3$
terminals of degree $3k$, one of which is chosen as the root) at every
terminal $t$ of the resulting graph. For the insertion, use an
arbitrary bijection between the respective edge sets. Each insertion
creates one induced copy of $K_{2,3k}$ (containing two terminals); the
vertex sets of these copies will be called \emph{atoms} of
$G_\ell$. We define a set $Y \subseteq V(G_\ell)$ to be \emph{aligned}
if every atom of $G_\ell$ is either contained in $Y$ or disjoint from
$Y$.

\begin{figure}
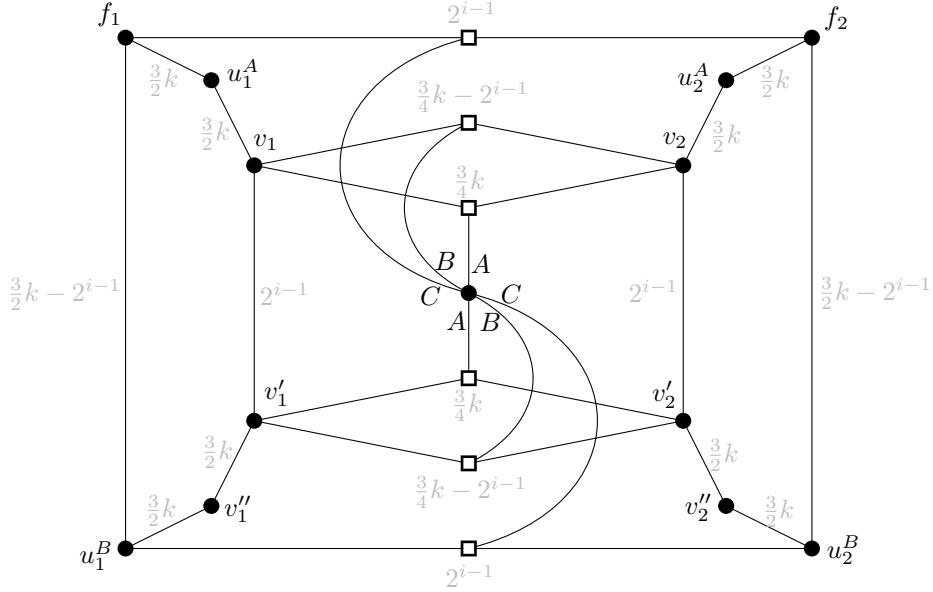

  \begin{center}
    \fig1
  \end{center}
  \caption{The graph with terminals $(F_i,T_i)$, where
    $1\leq i\leq \ell-1$. Terminals are shown by black dots. A white
    square with a grey numerical label represents an independent set
    of non-terminals of the given size, with each non-terminal joined
    by edges to all the terminals adjacent to the square mark. Labels
    on edges between terminals represent parallel edges of the given
    multiplicity. The root is the middle vertex. Labels on the
    adjacent edges represent the partition into the classes $A$, $B$,
    $C$.}
  \label{fig:fi}
\end{figure}

\begin{figure}
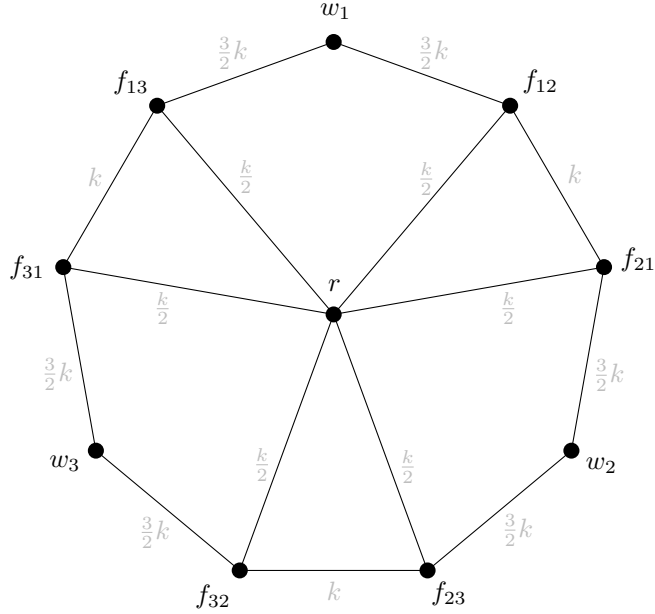

  \begin{center}
    \fig2
  \end{center}
  \caption{The graph with terminals $(F_\ell,T_\ell)$ when
    $\ell\geq 1$. The same conventions as in Figure~\ref{fig:fi}
    apply. The root is the middle vertex.}
  \label{fig:fl}
\end{figure}

As remarked before the construction, all terminals in the above graphs
have degree $3k$ and all non-terminals have degree $3$. Moreover, we
have the following consequence of Lemma~\ref{l:conn}.

\begin{observation}\label{obs:t-l-conn}
  The set $T_\ell$ is $3k$-edge-connected in $G_\ell$.
\end{observation}
\begin{proof}
  The set $T_i$ is $3k$-edge-connected in $F_i$ by direct inspection,
  for $0\leq i\leq\ell$. Similarly, the set of terminals is
  $3k$-edge-connected in the graph $N'$ created in step 5 of the
  construction (since $N$ is $3$-connected) and the same applies to
  $K_{3,3k}$ used in step 6. Thus, the observation follows from
  iterated use of Lemma~\ref{l:conn}.
\end{proof}

As promised, we now turn to the case $\ell=0$ (corresponding to
$k=1$). In this case, steps 1--4 may be skipped completely and we can
start with the $3$-connected non-hamiltonian cubic bipartite graph $N$
of step 5. We do not need to perform any insertions in this step, just
regard the colour class $X$ as terminals and the colour class $Y$ as
non-terminals. Step 6 is performed as described in the construction,
i.e., a copy of $K_{3,3}$ is inserted at each terminal.

%%%%%%%%%%%%%%%%%%%%%%%%%%%%%%%%%%%%%%%%%%%%%%%%%%%%%%%%%%%%%%%%%%%%%%

\section{Choosing the connectors}
\label{sec:choosing}

Recall that $k=2^\ell$ and consider the graph with terminals
$(G_\ell,T_\ell)$ constructed in
Section~\ref{sec:construction}. Recall that all terminals of $G_\ell$
have degree $3k$ and all non-terminals have degree $3$. As shown in
Observation~\ref{obs:t-l-conn}, $T_\ell$ is $3k$-edge-connected in
$G_\ell$.

An edge-cut $R$ in $G_\ell$ is said to be \emph{basic} if there is an
independent set $X_0 \subseteq V(G_\ell)$ consisting of non-terminals such
that the following holds for the partition $\PP$ of
$V(G_\ell)\setminus X_0$ into vertex sets of connected components of
$G_\ell-X_0$:
\begin{itemize}
\item for each $X\in\PP$, $d_{G_\ell}(X) = 3k$.
\item $R = \bdx{G_\ell}Y$ for some $Y\in \PP$,
\item each set in $\PP$ is aligned and contains at least one
  terminal,
\end{itemize}
In this situation, we say that $R$ is a basic edge-cut \emph{for
  $X_0$}.

\begin{observation}\label{obs:basic}
  If $X$ is an atom in $G_\ell$, then the edge-cut $\bdx{G_\ell}X$ is
  basic.
\end{observation}
\begin{proof}
  Let $X_0$ be the set of non-terminals in $X$. Observe that
  $G_\ell-X_0$ has precisely three components: two of them are
  singletons containing the terminals in $X$, and the third one is
  $G_\ell-X$. (The connectedness of $G_\ell-X$ follows from
  the fact that the graph entering step 6 of the construction of
  $G_\ell$ is 2-connected.) It is routine to check the conditions in the
  definition of a basic edge-cut.
\end{proof}

\begin{lemma}
  \label{l:deg-2k}
  Let $\ell\geq 0$ and let $k = 2^\ell$. Consider the graph with terminals
  $(G_\ell,T_\ell)$ and let $r\in T_\ell$. If
  Conjecture~\ref{conj:west-wu} holds for the given value of $k$, then
  there exist edge-disjoint $T_\ell$-connectors $Q_1,\dots,Q_k$ in
  $G_\ell$ with the following properties:
  \begin{enumerate}[label=(\roman*)]
  \item $Q := Q_1\cup\dots\cup Q_k$ contains exactly $2k$ edges in
    each basic edge-cut of $G_\ell$,
  \item $Q$ contains at least $2k$ edges of $\bdx{G_\ell}A$, where $A$ is
    any aligned set of vertices of $G_\ell-r$,
  \item for each $i\in [k]$, $Q_i-r$ is connected.
  \end{enumerate}
\end{lemma}
\begin{proof}
  Let $H$ be a $3$-connected cubic graph with more than $2k$
  vertices. Replace each edge of $H$ with $k$ parallel edges and
  regard all vertices of the resulting $3k$-regular graph $kH$ as
  terminals. Let ($H',T')$ be obtained by inserting, one by one, a
  copy of $(G_\ell,T_\ell)$ at each vertex of $kH$. The terminal $r$
  of $G_\ell$ from the statement of the lemma is used for the
  insertions, while the bijection between the edge sets is
  arbitrary. The copy of $G_\ell-r$ in $H'$ created by the insertion
  of $G_\ell$ at $w\in V(kH)$ will be denoted by $G^w$. Furthermore,
  in what follows, any atom in each such copy of $G_\ell-r$ will
  be referred to as an atom of $H'$.

  \begin{claim}
    The set $T'$ is $3k$-edge-connected in $H'$.
  \end{claim}
  \begin{claimproof}
    This follows directly from Lemma~\ref{l:conn} since $kH$ is
    clearly $3k$-edge-connected and $T_\ell$ is $3k$-edge-connected in
    $G_\ell$ by Observation~\ref{obs:t-l-conn}.
  \end{claimproof}
  
  Assuming the truth of Conjecture~\ref{conj:west-wu} for $k$, suppose
  that there exist $T'$-connectors $Q'_1,\dots,Q'_k$ in $H'$. Let these
  $T'$-connectors be chosen so as to
  \begin{enumerate}[label=(Q\arabic*)]
  \item minimise the total number of edges of $Q' := Q'_1\cup\dots\cup Q'_k$,
  \item subject to (Q1), maximise the total number of edges of $Q'$
    with both endvertices in the same atom.
  \end{enumerate}
  Observe that by minimality with respect to (Q1) and by the fact that
  each non-terminal in $H'$ has degree $3$, each $Q'_i$ is a tree.

  \begin{claim}\label{cl:atoms-conn}
    For any $i\in [k]$ and any atom $X$, the induced subgraph
    $Q'_i[X]$ contains a path joining the two terminals in $X$.
  \end{claim}
  \begin{claimproof}
    We prove the claim by contradiction. Suppose that the two
    terminals $x,y\in X$ are not joined by a path in $Q'_i[X]$. Let
    $P$ be the unique path in the tree $Q'_i$ joining $x$ to $y$, and
    let $P_0$ be the subpath of $P$ starting at $x$ and ending at the
    first terminal different from $x$. Let $u$ be the non-terminal
    adjacent to $x$ on $P_0$, and let $Q''_i$ be obtained by removing
    the edges of $P_0$ from $Q'_i$ and adding the edges
    $xu,uy$. Clearly, $Q''_i$ is a tree and a $T'$-connector; since
    the length of $P_0$ is at least $2$, $Q''_i$ has at most as many
    edges as $Q'_i$, and has more edges with both endvertices in $X$
    than $Q'_i$ does, contradicting the choice of $Q'_1,\dots,Q'_k$
    with respect to (Q2).
  \end{claimproof}
  
  \begin{claim}\label{atoms:deg}
    If $X$ is an atom of $H'$, then $d_{Q'}(X) \leq 2k$.
  \end{claim}
  \begin{claimproof}
    Let $x,y$ be the terminals of $X$. Since each $Q'_i$ ($i\in [k]$)
    contains an $xy$-path of length two by Claim~\ref{cl:atoms-conn}, at least $k$ of the edges in
    $\bdx{H'}X$ are not used by $Q'$. This proves the claim since
     $d_{H'}(X)=3k$.
  \end{claimproof}

  In the proofs of the following two claims, we will consider a
  modification of the connectors $Q'_i$. For $i\in [k]$, let
  $\tilde Q_i$ be the multigraph obtained from $Q'_i$ by the following
  procedure:
  \begin{itemize}
  \item identify all the vertices of each atom $X$ to a single vertex
    $v_X$ (keeping any parallel edges and discarding loops),
  \item suppress any remaining non-terminals whose degree in $Q'_i$ is
    $2$ and delete those whose degree in $Q'_i$ is $0$.
  \end{itemize}
  Note that all the multigraphs $\tilde Q_i$ are trees on the same
  vertex set, namely the set of vertices $v_X$ corresponding to the
  atoms $X$ of $H'$. Let $\tilde Q$ be the union of all $\tilde Q_i$
  over $i\in [k]$.
  
  \begin{claim}\label{cl:atom-2k}
    There is a vertex $w\in V(kH)$ such that if $X$ is any atom of
    $H'$ contained in $V(G^w)$, then $d_{Q'}(X) = 2k$.
  \end{claim}
  \begin{claimproof}
    Let $n$ be the number of vertices of the multigraph $\tilde
    Q$ defined above. Since $\tilde Q$ is the union of $k$ edge-disjoint trees on
    these $n$ vertices, it has $k(n-1)$ edges and therefore
    \begin{equation}\label{eq:lose-2k}
      \sum_{t\in V(\tilde Q)} d_{\tilde Q}(t) = 2k(n-1).
    \end{equation}
    On the other hand, Claim~\ref{atoms:deg} implies that for each
    $t\in V(\tilde Q)$, $d_{\tilde Q}(t) \leq 2k$. In view
    of~\eqref{eq:lose-2k}, there can be at most $2k$ vertices of
    $\tilde Q$ whose degree in $\tilde Q$ is strictly smaller than
    $2k$. Since $\size{V(H)} > 2k$, where $H$ is the cubic graph used
    to construct $H'$, the claim follows.
  \end{claimproof}
  
  \begin{claim}\label{cl:large}
    Let $Y$ be an aligned subset of $V(G^w)$, where $w$ is as in
    Claim~\ref{cl:atom-2k}. For $i\in [k]$, let $\kappa_i$ be the
    number of components of $Q'_i[Y]$ containing at least one edge
    each. Then
    \begin{equation}\label{eq:kappa}
      d_{Q'}(Y) = 2\cdot\sum_{i=1}^k\kappa_i.
    \end{equation}
    In particular, $d_{Q'}(Y) \geq 2k$.
  \end{claim}
  \begin{claimproof}
    Consider the multigraph $\tilde Q$ and let $D$ be the sum of
    $d_{\tilde Q}(X)$, taken over all atoms $X$ with $X\subseteq
    Y$. By Claim~\ref{cl:atom-2k}, $D = 2kn$, where $n$ is the number
    of such atoms. We evaluate $D$ in another way.

    Fix $i\in [k]$ and let $\delta_i$ be the number of edges of
    $\tilde Q_i$ with one endvertex corresponding to an atom in $Y$
    and one to an atom outside $Y$. Observe that
    $\sum_{i=1}^k \delta_i = d_{Q'}(Y)$. Furthermore, the number of
    edges of $\tilde Q_i$ with both endvertices corresponding to atoms in $Y$ is
    $n-\kappa_i$. It follows that
    \begin{equation*}
      D = \sum_{i=1}^k \Bigl(2\cdot (n - \kappa_i) + \delta_i\Bigr) =
      2kn - 2\sum_{i=1}^k \kappa_i + d_{Q'}(Y).
    \end{equation*}
    Since $D = 2kn$, equation~\eqref{eq:kappa} follows. In addition,
    since $\kappa_i \geq 1$ for each $i\in [k]$, the right hand side
    of~\eqref{eq:kappa} is greater than or equal to $2k$. This proves
    the claim.
  \end{claimproof}

  To obtain the sought $T_\ell$-connectors in $G_\ell$, take $w$ as in
  Claim~\ref{cl:atom-2k} and contract the complement of $V(G^w)$ in
  $H'$ to a single terminal $r$ of degree $3k$. The resulting graph
  with terminals is a copy of $(G_\ell,T_\ell)$. To simplify the
  notation, we will in fact identify it with $(G_\ell,T_\ell)$ in the
  rest of this proof. Note that by the construction of $H'$, the
  vertex $r$ obtained by the contraction is precisely the terminal of
  $G_\ell$ selected in the statement of the lemma.

  Define $Q_1,\dots,Q_k$ as the subgraphs of $G_\ell$ obtained from
  $Q'_1,\dots,Q'_k$ by this contraction (note that these are not
  necessarily trees) and let $Q$ denote their union. It is easy to see
  that each $Q_i$ ($i\in [k]$) is a $T_\ell$-connector in $G_\ell$.
  
  \begin{claim}\label{cl:basic}
    The graph $Q$ contains exactly $2k$ edges from each basic edge-cut
    in $G_\ell$.
  \end{claim}
  \begin{claimproof}
    Given a basic edge-cut $R$ in $G_\ell$ for a set $X_0$ of
    non-terminals, let $\PP$ be the partition of
    $V(G_\ell)\setminus X_0$ as in the definition of basic
    edge-cut, and let $p = \size{\PP}$. For any $X\in\PP$, the set
    $\bdx{G_\ell}X$ contains $3k$ edges, each of which is
    incident with a degree $3$ vertex in the independent set $X_0$. It
    follows that $\size{X_0} = kp$.

    We claim that for any $Y\in\PP$, $d_Q(Y)\geq 2k$. If $r\notin Y$,
    then this follows from Claim~\ref{cl:large} since $Y$ is then an
    aligned set of vertices of $G^w$, and we have
    $d_Q(Y) = d_{Q'}(Y)$. On the other hand, if $r\in Y$, then we
    apply Claim~\ref{cl:large} to the complement of $Y$ in
    $V(G_\ell)$ (which is also aligned) and get the same conclusion.

    Consequently, we obtain $\sum_{X\in\PP} d_Q(X) \geq 2kp$. On the
    other hand, each of the $kp$ vertices of $X_0$ has degree at most
    $2$ in $Q$, so equality holds and $d_Q(X) = 2k$ for any
    $X\in\PP$. This implies the claim.
  \end{claimproof}

  Properties (i) and (ii) in the lemma follow from
  Claim~\ref{cl:basic} and Claim~\ref{cl:large}, respectively. It
  remains to prove property (iii).

  We first prove that $Z := \bdx{G_\ell}{\Setx r}$ is a basic edge-cut
  in $G_\ell$. Let $N$ be the non-hamiltonian cubic bipartite graph
  used to construct $G_\ell$. Recall that $N$ has colour classes $X$
  and $Y$ and during the construction, each vertex of $Y$ is replaced
  by an independent set of non-terminals of degree $3$ and a copy of
  $F^*_\ell$ is inserted at each of the remaining vertices. We define
  $X_0$ as the set of non-terminals arising from the former operation,
  and observe that $Z$ satisfies the properties of a basic edge-cut
  for $X_0$.

  By Claim~\ref{cl:basic}, $Q$ contains $2k$ edges of $Z$. Thus,
  Claim~\ref{cl:large} (used with $Y = V(G^w)$) implies that for each
  $i\in [k]$, the induced subgraph of $Q'_i$ on $V(G_\ell)$ must have
  exactly one component --- that is, $Q_i-r$ is connected.
\end{proof}

%%%%%%%%%%%%%%%%%%%%%%%%%%%%%%%%%%%%%%%%%%%%%%%%%%%%%%%%%%%%%%%%%%%%%% 

\section{Proof of Theorem~\ref{t:main}}
\label{sec:proof}

In this section, we prove our main result.

\renewcommand{\thetheorem}{1.2}
\begin{theorem}
  For any integer $\ell \geq 0$, there are infinitely many graphs with
  terminals $(G,T)$ such that
  \begin{enumerate}[label=(\arabic*)]
  \item $T$ is $3\cdot 2^\ell$-edge-connected in $G$, and
  \item there do not exist $2^\ell$ edge-disjoint $T$-connectors in
    $G$.
  \end{enumerate}
\end{theorem}

\begin{proof}
  \setcounter{cclaim}0 Let $\ell\geq 0$ be given and let $k =
  2^\ell$. Let $(G_\ell,T_\ell)$ be the graph with terminals
  constructed in Section~\ref{sec:construction}. The construction uses
  a $3$-connected non-hamiltonian bipartite cubic graph $N$;
  by~\cite{EH}, there are infinitely many different choices of $N$,
  each producing a different graph $G_\ell$. Therefore, if we can show
  that $(G_\ell,T_\ell)$ obtained from a particular $N$ is a
  counterexample, then this implies an infinite class of
  counterexamples.

  Recall that $T_\ell$ is $3k$-edge-connected in $G_\ell$. Assuming
  that $G_\ell$ admits $k$ edge-disjoint $T_\ell$-connectors, let
  $Q_1,\dots,Q_k$ be $T_\ell$-connectors obtained using
  Lemma~\ref{l:deg-2k}, and let $Q$ be their union.

  Let us first assume that $\ell=0$ and therefore $k=1$, so the
  construction of Section~\ref{sec:construction} is much simplified,
  as described at the end of that section. In this case, we have
  $Q=Q_1$, where $Q_1$ is a $T_0$-connector. By Lemma~\ref{l:deg-2k}
  and Observation~\ref{obs:basic}, $Q$ contains exactly two edges in
  each set $\bdx{G_0}X$, where $X$ is an atom. By contracting each
  atom to a vertex (reversing step 6 of the construction), we thus
  change $Q$ to a $2$-regular spanning subgraph of $N$; since $Q$ is
  connected, it is a Hamilton cycle of $N$, contradicting the choice
  of $N$.

  In the rest of the proof, we assume that $\ell \geq 1$. We start by
  proving that certain edge-cuts in the graphs $(F_i,T_i)$ correspond
  to basic edge-cuts in $(G_\ell,T_\ell)$. Let $\widehat F$ be any
  specific copy of $F^*_i$ (with the root deleted and atoms inserted
  at each remaining terminal) in $G_\ell$. If $X$ is a set of vertices
  of $F_i$, then $\bdx{F_i}X$ does not exist in $G_\ell$ because of
  the insertions in steps 3 and 6 of the construction of $G_\ell$ in
  Section~\ref{sec:construction}. On the other hand, there is a
  well-defined set of edges of $G_\ell$ corresponding to $\bdx{F_i}X$
  in $\widehat F$, which we will denote by $\delta(X)$. (Strictly
  speaking, some of these edges will only have one endvertex in
  $\widehat F$ if they correspond to edges incident with the root of
  $F_i$.)

  We let $\varepsilon(X)$ denote the set of edges of $G_\ell$
  corresponding, in the same fashion, to the edges of the induced
  subgraph $F_i[X]$ on $X$. It will also be convenient to define
  \begin{equation*}
    \delta_Q(X) = \delta(X) \cap E(Q)\text{\quad and\quad}
    \varepsilon_Q(X) = \varepsilon(X) \cap E(Q).
  \end{equation*}
  In case that $X$ contains only one or two vertices, we omit the set
  brackets in these expressions, writing just $\delta(v)$ or
  $\varepsilon_Q(v,w)$.

\begin{claim}\label{cl:deg-Fi}
  Let $z$ be any terminal of $F_i$ (where $0\leq i\leq\ell-1$). Then
  $\delta(z)$ (as defined for the specific copy $\widehat F$ of
  $F^*_i$) is a basic edge-cut in $G_\ell$.
\end{claim}
\begin{claimproof}
  The claim is easy for a terminal $z$ such as $u^A_1$ or $v_1$, where
  the only insertion during the construction of $G_\ell$ is in step 6
  of the construction of $G_\ell$, namely the insertion of a
  $K_{3,3k}$ at $z$, creating an atom $X$. Now
  $\delta(z) = \bdx{G_\ell}X$, and by Observation~\ref{obs:basic},
  $\bdx{G_\ell}X$ is basic.

  This leaves only two cases to discuss: $f_1$ (and $f_2$, which is
  symmetric), and the root terminal $r$ of $F_i$. Suppose first that
  $z=r$. In the definition of basic edge-cut, choose $X_0$ as the set
  of all $3k$ non-terminals of $\widehat F$ corresponding to the
  non-terminals in $F_i$; it is easy to check that the requirements
  are fulfilled.

  Suppose next that $z=f_1$. If $i=0$, then the situation is similar
  as for $u^A_1$ above (the only insertion at $z$ creates an
  atom). Suppose that $i > 0$ consider the insertion of a copy of
  $F^*_{i-1}$ at $f_1$ in step 3. Let $r'$ denote the root of
  $F_{i-1}$, and define $\delta(r')$ as above, with respect to this
  insertion. Then $\delta(f_1) = \delta(r')$, which reduces the
  problem to the preceding situation. The case $z=f_2$ is symmetric.
\end{claimproof}

\begin{claim}\label{cl:cont-disj}
  Let $i\in [\ell-1]$ and let $\widehat F$ be a copy of $F^*_i$ (with
  the root $r$ deleted and atoms inserted at each remaining terminal)
  in $G_\ell$. Let $C$ be the set of all edges in
  $\bdx{G_\ell}{V(\widehat F)}$ corresponding to output edges of
  $F_i$.Then
  \begin{equation*}
    C\subseteq Q\text{\quad or\quad} C\cap Q = \emptyset.
  \end{equation*}
\end{claim}
\begin{claimproof}
  We prove this claim by induction on $i$. It is vacuous for $i=0$, so
  let us assume that $i > 0$. We refer to Figure~\ref{fig:fi} for the
  notation.

  By Claim~\ref{cl:deg-Fi} and property (i) of Lemma~\ref{l:deg-2k},
  \begin{equation}
    \label{eq:1}
    |\delta_Q(v'_1)| = |\delta_Q(v''_1)| = |\delta_Q(u^B_1)| = 2k.
  \end{equation}
  
  We want to show that
  \begin{equation}\label{eq:2}
    \size{\varepsilon_Q(v'_1,v''_1)} = k.
  \end{equation}
  Suppose that this is not the case and assume first that
  $\size{\varepsilon_Q(v'_1,v''_1)} < k$. Then $\size{\varepsilon_Q(v''_1,u^B_1)} > k$
  by~\eqref{eq:1}, so $|\delta_Q(u^B_1)\setminus\varepsilon_Q(v''_1,u^B_1)| < k$ again
  by~\eqref{eq:1}. Therefore,
  \begin{equation*}
    |\delta_Q(v''_1,u^B_1)| = |\varepsilon_Q(v'_1,v''_1)| +
    |\delta_Q(u^B_1)\setminus\varepsilon_Q(v''_1,u^B_1)| < 2k,
  \end{equation*}
  contradicting property (ii) of Lemma~\ref{l:deg-2k}. It follows that
  $\size{\varepsilon_Q(v'_1,v''_1)} > k$. Then $|\varepsilon_Q(v''_1,u^B_1)| < k$
  and we obtain an analogous contradiction for the set
  $\delta_Q(v'_1,v''_1)$. This proves~\eqref{eq:2}.

  Observe that the same argument proves~\eqref{eq:2} with $v''_1$
  replaced by any other terminal having only two neighbours in $F_i$
  and with $v'_1$ replaced by any of these neighbours.

  Next, we show
  \begin{equation}
    \label{eq:3}
    |\delta_Q(f_1)\setminus \delta_Q(u^A_1,u^B_1)| = |\delta_Q(u^B_1)\setminus \delta_Q(f_1,v''_1)|.
  \end{equation}
  Indeed, $|\delta_Q(f_1)| = |\delta_Q(u^B_1)| = 2k$ by Claim~\ref{cl:deg-Fi} and
  property (i) of Lemma~\ref{l:deg-2k}, and
  $|\varepsilon_Q(f_1,u^A_1)| = |\varepsilon_Q(u^B_1,v''_1)| = k$ by the
  extension of~\eqref{eq:2} mentioned above. This implies that
  $|\delta_Q(f_1)\setminus \delta_Q(u^A_1)| =
  |\delta_Q(u^B_1)\setminus \delta_Q(v''_1)| = k$,
  and~\eqref{eq:3} follows since it just refers to removing the same
  subset $\varepsilon_Q(f_1,u^B_1)$ from each of the sets in this
  equation.

  Let us return to the inductive proof of Claim~\ref{cl:cont-disj}. By
  the induction hypothesis, either all output edges of the copy of
  $F^*_{i-1}$ inserted at $f_1$ are contained in $Q$, or none of them
  are. In other words,
  \begin{equation}
    \label{eq:4}
    |\delta_Q(f_1)\setminus \delta_Q(u^A_1,u^B_1)| \in \Setx{0,2^{i-1}}.
  \end{equation}
  Thus, the two sides of~\eqref{eq:3} sum up to $0$ or $2^i$. By the
  left-right symmetry of Figure~\ref{fig:fi}, all of the above works
  with the subscript $1$ replaced by $2$; in particular, for $j =
  1,2$,
  \begin{equation}
    |\delta_Q(f_j)\setminus \delta_Q(u^A_j,u^B_j)| + |\delta_Q(u^B_j)\setminus
    \delta_Q(f_j,v''_j)| \in \Setx{0,2^i}.\label{eq:5}  
  \end{equation}

  Observe that if we know the value of the left hand side
  of~\eqref{eq:5} for $j=1,2$, we can infer the number of edges in
  $C\cap Q$. This is because $Q$ has degree $0$ or $2$ on every
  non-terminal of $G_\ell$. Thus, an edge $e\in C$ is contained in $Q$
  if and only if exactly one of the two edges incident with $e$ at the
  non-terminal vertex is contained in $Q$. In particular,~\eqref{eq:5}
  implies that the number of edges in $C\cap Q$ is either $0$ or $2^i$
  as claimed.
\end{claimproof}

A similar statement can be derived for the graph $F_\ell^*$.

\begin{claim}
  \label{cl:cont-disj-l}
  Let $\widehat F$ be a copy of $F^*_\ell$ (with the root $r$ deleted
  and atoms inserted at each remaining terminal) in $G_\ell$. Given
  $p,q$, where $1\leq p < q \leq 3$, let
  $C = \varepsilon_Q(r,f_{pq})\cup \varepsilon_Q(r,f_{qp})$ (where $r$, $f_{pq}$ and
  $f_{qp}$ are vertices of $F_\ell$, see Figure~\ref{fig:fl}). Then
  \begin{equation*}
    C\subseteq Q\text{\quad or\quad} C\cap Q = \emptyset.
  \end{equation*}
\end{claim}
\begin{claimproof}
  For simplicity, let us consider the case $(p,q) = (1,2)$. Similarly
  to~\eqref{eq:1}, we can show that
  $|\delta_Q(w_1)| = |\delta_Q(f_{12})| = |\delta_Q(f_{21})| = 2k$. Furthermore,
  $|\varepsilon_Q(w_1,f_{12})| = |\varepsilon_Q(w_2,f_{21})| = k$ is
  proved analogously to~\eqref{eq:2}. By combining these facts, we
  infer that
  \begin{equation}
    \label{eq:6}
    |\varepsilon_Q(r,f_{12})| = |\varepsilon_Q(r,f_{21})|.
  \end{equation}
  We know from Claim~\ref{cl:cont-disj} (applied to the copies of
  $F^*_{\ell-1}$ inserted at $f_{12}$ and $f_{21}$, respectively) that
  the set on the left hand side of~\eqref{eq:6} is empty or contains
  all edges of $G_\ell$ in $\varepsilon(r,f_{12})$. Thus,
  the present claim follows from~\eqref{eq:6} and symmetry.
\end{claimproof}

We have proved that for each $p,q$, where $1\leq p < q \leq 3$, $Q$
contains $0$ or $k$ edges in
$\varepsilon_Q(r,f_{pq})\cup \varepsilon_Q(r,f_{qp})$. Recall that this set of
edges corresponds to a single edge $e=xy$ of the cubic bipartite graph
$N$ used in step 5 of the construction of $G_\ell$ (where $x\in X$ is
a vertex corresponding to the given copy of $F^*_\ell$ in $G_\ell$,
and $y\in Y$ corresponds to an independent set of $k$ non-terminals in
$G_\ell$). Since $Q$ either contains all edges of $G_\ell$
corresponding to $e$ or none (for each $e$) we can associate with $Q$
a subgraph $Q_N$ of $N$. This subgraph is obtained from $Q$ by
identification of vertices, so $Q_N$ is a connected spanning subgraph
of $N$.

We will prove that $Q_N$ is $2$-regular. We claim that the edge-cut
$\bdx{G_\ell}{V(\widehat F)}$ is basic. To see this, recall that uses
a cubic graph $N$ with colour classes $X$ and $Y$; let $X_0$ be the
set of all non-terminals of $G_\ell$ coming from the vertices in $Y$,
and check that the said edge-cut is basic for $X_0$.

Lemma~\ref{l:deg-2k} therefore implies that $Q$ contains exactly $2k$
edges of $\bdx{G_\ell}{V(\widehat F)}$, and since $\widehat F$ is
arbitrary, $Q_N$ is $2$-regular. Being connected and spanning in $N$
at the same time, $Q_N$ is a Hamilton cycle of $N$. This contradicts
the choice of $N$ and finishes the proof of Theorem~\ref{t:main}.
\end{proof}

%%%%%%%%%%%%%%%%%%%%%%%%%%%%%%%%%%%%%%%%%%%%%%%%%%%%%%%%%%%%%%%%%%%%%% 

\section{The case of even $k$}
\label{sec:even}

In this section, we outline the changes that need to be made when $k$
is an arbitrary even positive integer (not necessarily a power of
two). Define
\begin{equation*}
  \ell = \lceil\log_2 k\rceil \text{\quad and \quad} x =
  2^{\ell-1}-\frac k2.
\end{equation*}
(In the preceding situation where $k$ is a power of two, the
definition of $\ell$ is consistent and $x$ is zero.) Note that
$2^{\ell-1} < k \leq 2^\ell$ and consequently $0\leq x < \frac k2$.

We are constructing a graph with terminals $(G,T)$, where $T$ is
$3k$-edge-connected in $G$ and no $3k$ edge-disjoint $T$-connectors
exist.

Steps 1--3 of the construction in Section~\ref{sec:construction} are
carried out with a single minor change: if $k$ is not divisible by
$4$, then the multiplicities $\frac34 k$ and $\frac34 k-2^{i-1}$
appearing in the graph with terminals $(F_i,T_i)$ in
Figure~\ref{fig:fi} (each twice), have to be rounded. We do this in
such a way that the vertex $v_1$ of $F_i$ is adjacent to
$\lceil\frac34 k\rceil$ non-terminals incident with class $A$ edges
and to $\lfloor\frac34 k-2^{i-1}\rfloor$ non-terminals incident with
class $B$ edges, and the rounding is performed in the opposite way for
the non-terminals adjacent to $v'_1$. In particular, the number of
class $A$ (class $B$) edges of $F_i$ is $\frac32k$ ($\frac32k-2^i$,
respectively), and each of $v_1, v'_1, v_2$ and $v'_2$ has degree
exactly $3k$ in $F_i$ as required.

As a result, we have the graph $F^*_{\ell-1}$ with terminals
$T^*_{\ell-1}$ at our disposal. The number of class $C$ edges incident
with the root in this graph (cf. Figure~\ref{fig:fi} for a picture of
$F_{\ell-1}$ is
\begin{equation*}
2(2^{\ell-1-1}) = 2^{\ell-1} = \frac k2+x.
\end{equation*}
Furthermore, the number of class $A$ and $B$ edges is $\frac{3k}2$ and
$k-x$, respectively.

In step 4 of the construction, we use the graph $(F_\ell,T_\ell)$ as
before, with edge multiplicities as shown in Figure~\ref{fig:fl}
(namely $\frac k2$, $\frac{3k}2$ and $k$, respectively). We need to be
careful when inserting copies of $F^*_{\ell-1}$ at each vertex
$f_{pq}$, where $1\leq p < q\leq 3$: while the number of class $A$
edges in $F^*_{\ell-1}$ is right (namely $\frac{3k}2$), there are too
few class $B$ edges and too many class $C$ edges. Let us say we are
inserting $F^*_{\ell-1}$ at $f_{12}$. We single out a set $X$ of $x$ class
$C$ edges (incident with the root $r$) in the given copy of
$F^*_{\ell-1}$, and perform the insertion using a bijection $h$
satisfying the following properties:
\begin{itemize}
\item $h$ maps the edges incident with the root $r$ of $F_\ell$ to the
  edges incident with the root of $F^*_{\ell-1}$ (of class $A$, $B$ or
  $C$),
\item $h$ maps the $\frac{3k}2$ edges between $r$ and $f_{12}$ to the
  class $A$ edges of $F^*_{\ell-1}$,
\item $h$ maps the $k$ edges between $f_{12}$ and $f_{21}$ to the
  class $B$ edges of $F^*_{\ell-1}$ together with the class $C$ edges
  in $X$,
\item $h$ maps the $\frac k2$ edges between $f_{12}$ and $w_1$ to the
  class $C$ edges of $F^*_{\ell-1}$ not in $X$.
\end{itemize}
The same procedure is performed for each of the sets $f_{pq}$ in turn.

Steps 5--6 are copied verbatim, except that the resulting graph with
terminals is denoted by $(G,T)$ rather than $(G_\ell,T_\ell)$. Note
that the analogue of Observation~\ref{obs:t-l-conn} applies, and the
subsequent proof goes through essentially without change.

%%%%%%%%%%%%%%%%%%%%%%%%%%%%%%%%%%%%%%%%%%%%%%%%%%%%%%%%%%%%%%%%%%%%%% 

% \section*{Acknowledgment}
% We thank the reviewers...

\end{document}